\documentclass[12pt]{amsart}
\usepackage[margin=1.2in]{geometry}
\usepackage{amsthm, amssymb, amsmath, mathrsfs, bm, enumerate, tikz, stmaryrd, epic, eepic, cite, cleveref}
\newtheorem{PROP}{Proposition}
\newtheorem{LMA}{Lemma}
\newtheorem*{DEF}{Definition}

\newtheorem{THM}{Theorem}
\newtheorem*{EXM}{Example}

\newtheorem{main}{Theorem}

\usepackage{enumitem}

\date{February 2023}
\begin{document}
\title{Geometric Conditions for Twisted $O$-sphericity}
\author{Arieh Zimmerman}
\address{Department of Mathematics and Computer Science, Weizmann Institute of Science, 
Rehovot, IL.}
\email{arieh.arbor@gmail.com}

\begin{abstract}The geometric condition defining a spherical variety for a reductive algebraic group was generalized in \cite{AG21}, with applications to representation theory. We twist by a character to generalize this definition, and show its equivalence to a property of group actions that generalizes Theorem B of \cite{AG21}. We also present an example to demonstrate the necessity of this generalization.
\end{abstract}

\maketitle

\section{The Twisted Condition}
Let $G$ be an algebraic group over an algebraically closed field $F$. Recall that, for an algebraic $G$-action on a smooth variety $X$ defined over $F$, we get a map $\nu:\mathfrak{g}\rightarrow\Gamma(X)$ to global vector fields on $X$. This in turn gives a map $\mu:T^*X\rightarrow\mathfrak{g}^*$ defined by $\mu(s_{(p,v)})(g)=s_{(p,v)}(\nu(g))$. It is then a classical theorem (see \cref{moment} below) that there are finitely many orbits if and only if $\dim\mu^{-1}(0)\leq\dim X$.

Classically, there is a notion of $X$ being spherical, meaning a Borel subgroup $B$ of $G$ has finitely many orbits on $X$.  Recasting the notion with this theorem, sphericity means that the moment map coming from the action of all of $G$ satisfies an inequality: $\dim\mu^{-1}(\mathfrak{b}^\perp)\leq\dim X$, where $\mathfrak{b}$ is the Lie algebra of $B$.

In \cite{AG21}, the authors generalize this to larger parabolic subgroups $P\subseteq G$. This uses its Richardson orbit $O_P$, which, if it has Lie algebra $\mathfrak{p}$ with unipotent radictal $\mathfrak{u}$, is the unique orbit whose intersection with $\mathfrak{u}$ is open and dense.

\begin{DEF}Such a $G$-variety $X$ is called $\overline{O_P}$-spherical if $\dim\mu^{-1}(O\cap\mathfrak{p}^\perp)\leq\dim X+\frac{1}{2}\dim O$ for all $O\subseteq\overline{O_P}$. When $X$ is a torsor $G/H$, we call the closed subgroup $H$ an $\overline{O_P}$-spherical subgroup.\end{DEF}

Note that, in this case where $X$ is this torsor, $\mu^{-1}(O)$ is isomorphic to $G\times_H(\mathfrak{h}^\perp\cap O)$, and the inequality can be rewritten $\dim(O\cap\mathfrak{h}^\perp)\leq\frac{1}{2}\dim O$. 

\begin{THM}[\cite{AG21}, Theorem B]\label{untwisted}$P$ has finitely many orbits on $X$ if and only if $X$ is $\overline{O_P}-spherical$.\end{THM}

This theorem provides an accessible way to determine $\overline{O_P}$-sphericity, the importance of which is demonstrated by the main result in \cite{AG21}. It continues, for the case $F=\mathbb{C}$, a string of results on finite multiplicities resulting from sphericity, and partially generalizes \cite{KO}, where smooth functions on $G/H$ were shown to have bounded multiplicities as a $G(\mathbb{R})$-representation if $H$ is spherical.

\begin{THM}[\cite{AG21}, Theorem C]Assume $G,H$ are defined over $\mathbb{R}$, and let $H$ be $\overline{O_P}$-spherical. For any finitely generated smooth admissible Frechet representation $\pi$ of moderate growth with wavefront set contained in $\overline{O_P}$, we have$$\dim\textrm{Hom}_{G(\mathbb{R})}(\mathcal{S}(G/H(\mathbb{R})),\pi)<\infty.$$Here $\mathcal{S}(G/H(\mathbb{R}))$ is the space of Schwartz functions on $\mathbb{R}$-points of $G/H$.\end{THM}

The goal of this note is a generalization of \cref{untwisted}, with the hope that such a finite multiplicity result can hold with weaker hypotheses. After experimentation yielded an example where finite multiplicities are likely to hold (see \cref{example} below), but $\overline{O_P}$-sphericity did not, the author's advisor took the essential characteristics of this example and proposed the following definition. Let $H\subseteq G$ be a connected closed subgroup of an algebraic group over $\mathbb{C}$, $P\subseteq G$ be parabolic, and $O_P\subset\mathfrak{g}^*$ be its Richardson orbit. We will also choose a character $\chi\in\mathfrak{h}^*$. We are going to pull back this character under the restriction $r:\mathfrak{g}^*\rightarrow\mathfrak{h}^*$.
\begin{DEF}The pair $(H,\chi)$ will be called $O$-spherical if$$\dim(O\cap r^{-1}(\chi))\leq\dim O/2.$$We call it $\overline{O_P}$-spherical if it is $O$-spherical for every $O\subseteq\overline{O_P}$.\end{DEF}
Note that for $\chi=0$, by Lemma 2.2.2 of \cite{AG21}, this reduces to the $O_P$-sphericity of $H$. In that case $H$ has finitely many orbits on $G/P$. Along these lines, our goal is to extend the methods of \cite{AG21} to prove the following generalization of \cref{untwisted}:
\begin{main}\label{result}Let $Y=\{x\in G/P:\chi|_{\textrm{stab}_{\mathfrak{h}}(x)}=1\}$. Then $H$ has finitely many orbits on $Y$ if and only if $(H,\chi)$ is $\overline{O_P}$-spherical.\end{main}

Using the theorem, we found examples of potential interest that satisfy twisted sphericity but require the twist.

\begin{EXM}\label{example}Define a pair$$GL_{6}(\mathbb{C})=:G\supseteq H:=\left\{\left(\begin{array}{ccc}B&X&Y\\0&B&Z\\0&0&B\end{array}\right)\Bigg|\ B=\left(\begin{array}{cc}b_1&b_2\\0&b_3\end{array}\right)\in GL_{2}\right\}$$acting on the almost-complete flag variety $G/P:=GL_6/P_{2,1,1,1,1}$. Then put $\chi(\log h)=X+Y$ for $h\in H$ written as above (sufficiently close to $\textrm{Id}$). Then $(H,\chi)$ is $\overline{O_P}$-spherical.\end{EXM}

That this example satisfies the hypotheses of the main theorem is carried out in \cite{thesis}.

\section{A Proof via Symplectic Manifolds}
We will need some conventions. Take some $y\in Y$, and let's say its preimage under the quotient $G\rightarrow G/P$ is the parabolic $Q$, conjugate to $P$. Since the algebra of the stabilizer of $Q$ in $H$ is $\mathfrak{q}\cap\mathfrak{h}$, the definition of $Y$ means that $\lambda\in(\mathfrak{q}\cap\mathfrak{h})^\perp=\mathfrak{h}^\perp+\mathfrak{q}^\perp$ for any $\lambda$ restricting to $\chi$ on $\mathfrak{h}$. When the choice of parabolic conjugate is clear, we will accordingly write $\lambda=\lambda'+\lambda''$.
\par Next we reformulate the condition that $H$ has finitely many orbits on $Y\subset G/P$. The $H$ orbits on $G/P$ are in bijection with the $G$ orbits on $G/H\times G/P$ by the map $G\cdot(g_1H,g_2P)\mapsto H\cdot g_1^{-1}g_2P$. To restrict to those orbits in the $H$-subspace $Y$ is to consider only the $G$-subspace defined as$$G/H\times G/P\supseteq W=\{(g_1H,g_2P)\mid\chi\mid_{\mathfrak{h}\cap g_1^{-1}g_2\mathfrak{p}g_2^{-1}g_1}\equiv0\}.$$So by this translation we are dealing with the finiteness of $G$-orbits in $W$. One advantage to using orbits under $G$ is the following generalization of a classical proposition:
\begin{LMA}\label{moment}For a smooth $G$-space $S$, let $f:T^*S\rightarrow S$ be the natural projection and $W\subseteq S$ be a locally closed $G$-subspace. Then $G$ has finitely many orbits on $W$ if and only if $\dim f^{-1}(W)\cap\mu_G^{-1}(0)\leq\dim S$.\end{LMA}\begin{proof}Each orbit lying in $W$ has a conormal bundle that is irreducible and Lagrangian, hence has dimension $\frac{1}{2}\dim T^*S=\dim S$. The union of these is precisely $f^{-1}(W)\cap\mu_G^{-1}(0)$.\par If $G$ has finitely many orbits on $W$, then the union is finite, and its dimension is also $\dim S$.\par Assume instead that there are infinitely many orbits. By Proposition 1.26 in \cite{Br}, there is a stratification $W=U_1\sqcup\dots\sqcup U_k$ by $G$-invariant subsets admitting a geometric quotient (here $U_{i+1}$ is the set of stable points in $U_i$). Some $U_i$ has infinitely many orbits under $G$, so the dimension of $f^{-1}(U_i)\cap\mu_G^{-1}(0)$ is strictly larger than the dimension of the conormal bundle of a single orbit. This means $\dim f^{-1}(W)\cap\mu_G^{-1}(0)>\dim S$.\end{proof}
From here, one direction of our desired theorem can be accomplished with only the use of a lemma.
\begin{LMA}\label{Chevalley}For two surjective morphisms $f_i:X_i\rightarrow Y$ of algebraic sets with $\dim f_1^{-1}(y)=\dim f_2^{-1}(y)$ at every $y\in Y$, we get $\dim X_1=\dim X_2$.\begin{proof}See e.g.\ Lemma 2.1 in \cite{holonomicity}.\end{proof}\end{LMA}
This lemma is used to compare subsets of the cotangent bundle to get the reverse direction in \cref{result}.
\begin{PROP}If $(H,\chi)$ is $\overline{O_P}$-spherical, then $H$ has finitely many orbits on $Y$.\end{PROP}\begin{proof}We will soon show that $W\subseteq S=G/H\times G/P$ is constructible. Therefore it is a union of locally closed sets, each of which is $G$-invariant. By \cref{moment}, it suffices to show $\dim f^{-1}(W)\cap\mu_G^{-1}(0)\leq\dim S$. We begin by decomposing the fiber according to coadjoint orbits.For $O\subset\overline{O_P}$, we can define $T_O$ as$$\{((g_1,\xi),(g_2,\eta))\in T^*G/H\times T^*G/P\mid g_1\cdot\xi=-g_2\cdot\eta\in O\}.$$The notation here regards the moment map $T^*G/H\rightarrow\mathfrak{g}^*$ equivalently as the function on $G\times_H\mathfrak{h}^\perp$ sending $(g_1,\xi)\mapsto g_1\cdot\xi$, and likewise with $P$. In this way we get a decomposition$$\mu^{-1}(0)=\bigcup_{O\in\overline{O_P}}T_O,$$and there are finitely many such $T_O$ since there are finitely many nilpotent orbits \cite{CM}. So it will suffice to demonstrate from $O$-sphericity that $\dim f^{-1}(W)\cap T_O\leq\dim S$.\par Rather than find the dimension of $T_O$, we do so for a twisted version of $T_O$ that we call $T_{O,\lambda}$, defined similarly as$$\{((g_1,\xi),(g_2,\eta))\in T^*W\subseteq T^*G/H\times T^*G/P\mid g_1\cdot(\xi-\lambda)=-g_2\cdot\eta\in O\}.$$Note first that $f(T_{\lambda,O})=W$. This is because\begin{align*}\chi(\mathfrak{h}\cap g_1^{-1}g_2\mathfrak{p}g_2^{-1}g_1)&=\xi(\mathfrak{h}\cap g_1^{-1}g_2\mathfrak{p}g_2^{-1}g_1)\\&=-g_1^{-1}g_2\cdot\eta(\mathfrak{h}\cap g_1^{-1}g_2\mathfrak{p}g_2^{-1}g_1)=\eta(\mathfrak{p})=0.\end{align*}Since $f$ is finitely presented and $T_{O,\lambda}\subset S$ is constructible, this fulfills the promise that $W$ is constructible \cite{M}. Our plan is to show that the dimension of $T_O$ is the same as that of $T_{O,\lambda}$. Here we use \cref{Chevalley}, applied to the projections $F:f^{-1}(W)\cap T_O\rightarrow W$ and $f:T_{O,\lambda}\rightarrow W$ onto our base space $W$. To apply it, we need only show that the fibers of a point in the image - the intersections with a cotangent space at a fixed $(g_1H,g_2P)$ - have equal dimensions. At this fixed base point, we write $\lambda=\lambda'+\lambda''\in\mathfrak{h}^\perp+(g_1^{-1}g_2\mathfrak{p}g_2^{-1}g_1)^\perp$ as in the conventions set earlier.\par We claim that the addition of $\lambda$ from $\xi$ is an isomorphism of these fibers. Indeed,$((g_1,\xi),(g_2,\eta))\in T_O$ means that $\xi=-g_1^{-1}g_2\cdot\eta\in O$, and the space of such pairs is $\mathfrak{h}^\perp\cap-g_1^{-1}g_2\mathfrak{p}^\perp g_2^{-1}g_1\cap O$. Adding $\lambda''$ to this, we get the space $(\mathfrak{h}^\perp+\lambda'')\cap-g_1^{-1}g_2\mathfrak{p}^\perp g_2^{-1}g_1\cap(O+\lambda'')$, and this is precisely the space of $\xi-\lambda'$ for which$$\mathfrak{h}^\perp\ni(\xi-\lambda')-\lambda''=\xi-\lambda=-g_1^{-1}g_2\cdot\eta\in O.$$So what is left is a careful counting argument to find the dimension of $T_{O,\lambda}$. Because the action of $G$ is diagonal, projection to the second coordinate of $T_{O,\lambda}$ is $G$-equivariant. It is also well-known that the moment map for $G/P$ is $G$-equivariant. Thus we can find dimensions by first fixing a point in $O$ and calculating the dimension of its fiber.\par Let's say that $\alpha\in O$ is fixed. Then the space of $\eta\in\mathfrak{p}^\perp$ for which we can present $\alpha$ as some $g_2\cdot\eta$ is $O\cap\mathfrak{p}^\perp$, since $G/P\cdot\alpha=O$. The arbitrariness of our choice of $g_2\in G/P$ adds another dimension $\dim\textrm{Stab}_{G/P}O=\dim G/P-\dim O$.\par The calculation for the first coordinate is similar. We have $O\cap(\mathfrak{h}^\perp+\lambda)$ ways to choose $\xi$ that allow us to write $\xi-\lambda=g_1^{-1}\cdot\alpha$, and $\dim G/H-\dim O$ degrees of freedom in stabilizing this $\xi-\lambda$. Adding this together, the dimension of our space fibered over $O$ is\begin{align*}\dim T_{O,\lambda}&=\\\dim O+\dim O\cap\mathfrak{p}^\perp&+\dim G/P-\dim O+\dim O\cap(\mathfrak{h}^\perp+\lambda)+\dim G/H-\dim O\\&\leq\frac{1}{2}\dim O+\dim G/P+\frac{1}{2}\dim O+\dim G/H-\dim O\\&=\dim G/H+\dim G/P=\dim S.\end{align*}Here we invoked $O$-sphericity, as well as the fact that $\dim O\cap\mathfrak{p}^\perp$ is an isotropic subspace of $O$ \cite{CG} and so has at most half the dimension.\end{proof}
Now we would like the analogue of the cotangent bundles used in \cite{AG21}. Being cotangent bundles as such was not as important as the existence of a symplectic structure, and something akin to the moment map.
\begin{PROP}[\cite{CG}, Proposition 1.4.14]\label{symplectic}For any $\lambda\in[\mathfrak{h},\mathfrak{h}]^\perp\subset\mathfrak{g}^*$, the space $\mathfrak{h}^\perp+\lambda$ is $H$-invariant. Moreover, there is a natural $G$-invariant symplectic form $\omega$ on $G\times_H\left(\mathfrak{h}^\perp+\lambda\right)$ given by\begin{enumerate}\item$\omega(\alpha,\beta)=0$ for $\alpha,\beta\in\mathfrak{h}^\perp$ tangent to the fibers of the projection $G\times_H(\mathfrak{h}^\perp+\lambda)\rightarrow G/H$.\item$\omega(v_x,v_y)\mid_\lambda=\lambda(g[x,y]g^{-1}])$ for any point $(g,\zeta)$ and vector fields $v$ coming from $x,y\in\mathfrak{g}$.\item$\omega(\alpha,v_x)=\alpha(gxg^{-1})$ at all $\alpha\in\mathfrak{h}^\perp$ tangent to the fiber at $gH\times(\mathfrak{h}^\perp+\lambda)$.\end{enumerate}\end{PROP}

Let us now complete the proof of \cref{result}.
\begin{PROP}If $H$ has finitely many orbits on $Y$, then $(H,\chi)$ is $\overline{O_P}$-spherical.\end{PROP}\begin{proof}Let $O\subseteq\overline{O_P}$. In light of \cref{symplectic}, we give a different presentation of $T_{O,\lambda}$ to take advantage of a symplectic structure: it will now be considered as the submanifold$$\left\{((g_1,\zeta),(g_2,\eta))\in(G\times_H(\mathfrak{h}^\perp+\lambda))\times T^*G/P\mid g_1\cdot\zeta=-g_2\cdot\eta\in O\right\}.$$When we restrict the map $(g_1,\zeta)\mapsto g_1\cdot\zeta$ to have its range in $O$, the graph will be$$\Gamma_{O,\lambda}=\{((g_1,\zeta),g_1\cdot\zeta)\}\subset(G\times_H(\mathfrak{h}^\perp+\lambda))\times O.$$This graph has an equivariant projection to $G/H$ with fiber $O\cap(\mathfrak{h}^\perp+\lambda)$ at $1$, making its dimension $\dim O\cap(\mathfrak{h}^\perp+\lambda)$.\par Consider the map $(I,\nu):T_{O,\lambda}\rightarrow\Gamma_{O,\lambda}$ , where $\nu$ is the moment map $T^*G/P\rightarrow\mathfrak{g}^*$. Note that $\nu$ is $G$-equivariant, so because its image on $\nu^{-1}(O)$ is a single orbit, it is surjective onto $O$. Any point in $\Gamma_{O,\lambda}$ can therefore be found in the image of some point in $T_{O,\lambda}$ under $(I,\nu)$. So $(I,\nu)$ is surjective and equivariant, hence a submersion, and moreover its codifferentials preserve the symplectic form on $(G\times_H(\mathfrak{h}^\perp+\lambda))\times G/P$ by Proposition 2.2.3 of \cite{AG21}. So our strategy will be to show that the $G$-space $T_{O,\lambda}$ is isotropic, for this would then prove that the graph $\Gamma_{O,\lambda}$ is also isotropic and give the dimension formula\begin{multline*}\dim O\cap(\mathfrak{h}^\perp+\lambda)+\dim G/H=\dim\Gamma_{O,\lambda}\leq\frac{1}{2}\dim(G\times_H(\mathfrak{h}^\perp+\lambda))\times O\\=\frac{1}{2}(\dim G-\dim H+\dim\mathfrak{h}^\perp+\dim O)=\dim G/H+\frac{1}{2}\dim O.\end{multline*}This is $O$-sphericity.\par To show that $T_{O,\lambda}$ is isotropic, define, for a $G$-orbit $C\subseteq W$, the subspace $T_C=f^{-1}(C)$ (here again $f$ is projection to the base). Since there are finitely many orbits, $T_{O,\lambda}$ would then be a finite union of such $T_C$'s, and it therefore suffices to show that such a $T_C$ is isotropic.\par Let $((g_1,\zeta),(g_2,\eta))\in T_C$. Then the tangent space to the homogeneous space $C$ is spanned by the vectors that come from the diagonal infinitesimal action of $G$. At a covector with basepoint $c$, if this action yields the vector $v_x$, then the tangent vectors of $T_C$  are of the form $(v_x+\alpha,v_x+\beta)$ with $\alpha,\beta$ vertical i.e.\ tangent solely to $f^{-1}(c)$. The condition $g_1\cdot\zeta=-g_2\cdot\eta$ differentiates with respect to $\zeta,\eta$ to give $g_1\cdot\alpha=-g_2\cdot\beta\ \forall\alpha,\beta$. Now taking the symplectic form $\omega_1$ of $G\times_H(\mathfrak{h}^\perp+\lambda)$ together with the symplectic form $\omega_2$ of $T^*G/P$ for two vectors tangent to our submanifold, we see\begin{align*}\omega_1(v_{x_1}+\alpha_1,v_{x_2}+\alpha_2)+\omega_2(v_{x_1}+\beta_1,v_{x_2}+\beta_2)&\\=\zeta(g_1[x_1,x_2]g_1^{-1})+\eta(g_2[x_1,x_2]g_2^{-1})&+\\\alpha_2(g_1x_1g_1^{-1})+\beta_2(g_2x_2g_2^{-1})-\alpha_1(g_1x_2g_1^{-1})&-\beta_1(g_2x_2g_2^{-1})\\=(g_1\cdot\zeta)([x_1,x_2])+(g_2\cdot\eta)([x_1,x_2])&+\\(g_1\cdot\alpha_2)(x_1)+(g_2\cdot\beta_2)(x_2)&-\\\left((g_1\cdot\alpha_1)(x_2)+(g_2\cdot\beta_1)(x_2)\right)\\=0-0+0.\end{align*}\end{proof}

\section*{Acknowledgments}
This is an abridgement of a thesis submitted for the degree of Master of Science at the Weizmann Institute of Science. Having the good fortune to be advised by Dima Gourevitch has been heartening and inspiring. His guidance was uniquely well-mannered for patience, support, respect, and openness.

\bibliographystyle{plain}
\bibliography{Biblio}

\end{document}